\documentclass[11pt,reqno]{amsart}
\usepackage{amsthm,amssymb,amsmath}

\usepackage[T1]{fontenc}
\usepackage{lmodern}
\usepackage{hyperref}
\usepackage{graphicx}
\usepackage{fix-cm}
\newtheorem{theorem}{Theorem}
\newtheorem*{theorem*}{Theorem}

\newtheorem{claim}{Claim}

\theoremstyle{definition}

\newtheorem{definition}{\sc Definition}
\newtheorem*{definition*}{\sc Definition}

\newtheorem*{remark*}{Remark}
\newtheorem*{remarks}{Remarks}

\newtheorem*{example*}{\bf Example}
\newtheorem*{examples}{\bf Examples}

\newcommand{\loc}{{\rm loc}}

\newcommand{\Real}{{\rm Re}}

\makeatletter
\newcommand{\dotminus}{\mathbin{\text{\@dotminus}}}

\newcommand{\@dotminus}{%
  \ooalign{\hidewidth\raise1ex\hbox{.}\hidewidth\cr$\m@th-$\cr}%
}
\makeatother

\begin{document}

\fontsize{10.5pt}{4.5mm}\selectfont

\title {Remarks on parabolic Kolmogorov operator}

\author{D.\,Kinzebulatov} 

\address{Universit\'{e} Laval, D\'{e}partement de math\'{e}matiques et de statistique, Qu\'{e}bec, QC, Canada}

\email{damir.kinzebulatov@mat.ulaval.ca}

\author{Yu.\,A.\,Sem\"{e}nov}

\address{University of Toronto, Department of Mathematics, Toronto, ON, Canada}

\email{semenov.yu.a@gmail.com}

\thanks{The research of D.K.\,is supported by the NSERC grant (RGPIN-2024-04236)}

\begin{abstract}
We obtain gradient estimates on solutions to parabolic Kolmogorov equation with singular drift in a large class. Such estimates allow to construct a Feller evolution family, which is used to construct unique weak solution to the corresponding stochastic differential equation.
\end{abstract}



\maketitle

\section{Introduction}

The paper is concerned with Sobolev regularity of solutions of the backward Kolmogorov equation
\begin{equation}
\label{kolm_eq}
\bigg(\partial_t-\Delta + b(t,x)\cdot \nabla\bigg)u=0,
\end{equation}
under some minimal assumptions on a locally unbounded vector field $b:\mathbb  R^{1+d} \rightarrow \mathbb R^d$, $d \geq 3$ (Theorem \ref{thm1}). These Sobolev regularity estimates allow us to construct, using a parabolic Moser-type iteration procedure, the corresponding to \eqref{kolm_eq} strongly continuous Feller evolution family. In turn, this Feller evolution family determines, for every $x \in \mathbb R^d$,
a unique in a large class weak solution to stochastic differential equation (SDE)
\begin{equation}
\label{seq}
X_t=x-\int_0^t b(s,X_s)ds + \sqrt{2}B_t, \quad B_t \text{ is a $d$-dimensional Brownian motion}, \quad x \in \mathbb R^d,
\end{equation}
see discussion and Theorems \ref{thm2}, \ref{thm3} below.

\medskip

The question of what minimal assumptions on a locally unbounded vector field $b$ provide various regularity properties of solutions of \eqref{kolm_eq} arises in many applications of parabolic partial differential equations, ranging from interacting particle systems to hydrodynamics. 
 It is practically impossible to survey the corresponding literature with any degree of completeness. We refer to the monograph \cite{BKRS}, see also discussion and results in \cite{NU}. That said, there are rather few results on the Sobolev regularity of $u$ i.e.\,estimates on $\|u\|_{L^\infty([0,1],W^{1,q})}$, where $W^{1,q}=W^{1,q}(\mathbb R^d)$ is the Sobolev space and $q>2$ is assumed to be large, despite their importance. Such estimates are the subject of this work.

\medskip

Perhaps the most popular condition on locally unbounded $b$ found in the literature on Kolmogorov equation \eqref{kolm_eq} is the Ladyzhenskaya-Prodi-Serrin class:
\begin{equation}
\label{LPS}
\tag{LPS}
|b| \in L^r(\mathbb R,L^p(\mathbb R^d)), \quad \frac{2}{r}+\frac{d}{p} \leq 1, \quad r \geq 2, \;\;p \geq d
\end{equation}
(furthermore, many results need more restrictive assumption $\frac{2}{r}+\frac{d}{p}<1$).
The \eqref{LPS} class is optimal in the sense that condition $\frac{2}{r}+\frac{d}{p} \leq 1$ cannot be replaced  by $\frac{2}{r}+\frac{d}{p} \leq 1+\varepsilon$ regardless of the smallness of $\varepsilon>0$  without destroying  H\"{o}lder continuity of solution $u$, not to mention its Sobolev regularity. That said, one can argue that the vector fields in the Ladyzhenskaya-Prodi-Serrin have rather mild singularities: given a vector field $b$ satisfying \eqref{LPS}, one can replace it in \eqref{kolm_eq} with $cb$  for an arbitrarily large constant $c>0$ without leaving the \eqref{LPS} class and without affecting the Sobolev regularity estimates on $u$. 

\medskip

In the present paper we consider a larger class of vector fields $b$. It contains some vector fields that are strictly more singular than the ones in \eqref{LPS}, such as
Hardy vector field
\begin{equation}
\label{hardy}
b(x)=\sqrt{\delta} \frac{d-2}{2}|x|^{-2}x, \quad \delta>0.
\end{equation}
Now, replacing $b$ by $cb$ for some $c>0$ or, equivalently, adjusting $\delta$, affects admissible $q>2$ in the estimates on $\|u\|_{L^\infty([0,1],W^{1,q})}$, see Theorem \ref{thm1}. In other words, the Sobolev regularity estimates on $u$ ``sense'' the change in the magnitude of the singularities of $b$. The latter is measured by constant $\delta$. 
Our focus is on describing the dependence of admissible $q$ in the estimates on $\|u\|_{L^\infty([0,1],W^{1,q})}$ on  $\delta$.

\medskip

Specifically, we consider the following class of vector fields:

\begin{definition}
A Borel measurable vector field $b:\mathbb R^{1+d} \rightarrow \mathbb R^d$ is said to be form-bounded if
$$|b| \in L^2_{\loc}(\mathbb R^{1+d})$$ 
and there exist constant $\delta>0$  and function $0\leq g=g_\delta\in L^1_\loc(\mathbb R)$ such that for all $t>0$
\begin{equation}
\label{fbd}
\int_0^t \|b(\tau) f(\tau) \|_2^2d\tau \leq \delta \int_0^t \|\nabla f(\tau)\|_2^2d\tau  + \int_0^t g(\tau) \|f(\tau) \|_2^2d\tau 
\end{equation}
for all $f \in \mathcal S(\mathbb R^{1+d})$.
\end{definition}

Here and everywhere below, $\mathcal S(\mathbb R^{1+d})$ is the Schwartz space of rapidly decreasing functions (Section \ref{notations}), and $$\|\varphi(\tau)\|^2_2:=\int_{\mathbb{R}^d}|\varphi(\tau,x)|^2dx, \quad \|\nabla \varphi(\tau)\|_2^2=\int_{\mathbb{R}^d}|\nabla_x \varphi (\tau,x)|^2dx.$$

We abbreviate \eqref{fbd} as $b\in \mathbf{F}_{\delta,g}$. The constant $\delta$ is called a form-bound of $b$.

\begin{examples}
1.~Every vector field in the \eqref{LPS} class belongs to $\mathbf{F}_{\delta,g}$ with $\delta$ that can be chosen arbitrarily small, as can be seen using Sobolev's embedding theorem and H\"{o}lder inequality. 

2.~The Hardy vector field \eqref{hardy} is in $\mathbf{F}_{\delta,0}$, as follows from the well known Hardy inequality $\||x|^{-1}\varphi\|_2 \leq \frac{4}{(d-2)^2}\|\nabla \varphi\|_2^2$, $\varphi \in W^{1,2}$ (this inequality is sharp, i.e.\,\eqref{hardy} is not in $\mathbf{F}_{\delta',g}$ for any $\delta'<\delta$ regardless of the choice of $g=g_{\delta'}$).

3.~Further, the class of form-bounded vector fields contains $L^\infty(\mathbb R,L^{d,\infty})$, where $L^{d,\infty}$ is the weak $L^d$ class. More generally, $L^{d,\infty}$ can be replaced with large scaling-invariant Morrey classes. In these classes, and hence in $\mathbf{F}_{\delta,g}$ there are $b=b(x)$ that are not in $L^{2+\varepsilon}_{\loc}$ for any $\varepsilon>0$. The proofs of these inclusions and some other examples can be found in \cite{Ki_survey, KiM, KiS_super,S}.

\end{examples}

The question of admissible values of form-bound $\delta$ is important, 
in particular, in the light of the following counter-example and result:

\medskip

(a) Let $b$ be the Hardy  drift \eqref{hardy}.
If $$\delta>4\bigg(\frac{d}{d-2}\bigg)^2,$$ then SDE \eqref{seq} with initial point $x=0$ does not have a weak solution, see e.g.\,proof in \cite{BFGM}. Informally, in the struggle between the diffusion and the drift the latter starts to have an upper hand.

\medskip

(b) On the other hand, SDE \eqref{seq} with $b \in \mathbf{F}_{\delta,g}$ (which includes \eqref{hardy}) has a unique in a large class weak solution for every $x \in \mathbb R^d$ provided that $\delta$ is sufficiently small. See \cite{KiM} or Theorem \ref{thm3} below.

\medskip

It should be added that SDE \eqref{seq2} with $b \in \mathbf{F}_{\delta,g}$ has a martingale solution for every initial point $x \in \mathbb R^d$ provided only 
\begin{equation}
\label{delta_4}
\delta<4, 
\end{equation}
see \cite{KiS_sharp}. 
Theorems \ref{thm2}, \ref{thm3} below, on the other hand, contain a more detailed weak well-posedness result for \eqref{seq} although this comes at expense of requiring smaller $\delta$. The proofs of both of these theorems use the gradient bounds of Theorem \ref{thm1}. 

\medskip

In the time-homogeneous case $b=b(x)$, analogous to Theorem \ref{thm1} gradient estimates on solution of elliptic equation $(\mu-\Delta + b \cdot \nabla v)=f$ were obtained in \cite{KS}. These gradient estimates allowed the authors of \cite{KS} to verify conditions of the Trotter approximation theorem in $C_\infty$ (continuous functions on $\mathbb R^d$ vanishing at infinity endowed with the $\sup$-norm) using a Moser-type iteration procedure, thus obtaining the corresponding to \eqref{kolm_eq} strongly continuous Feller semigroup in $C_\infty$.
Working in the elliptic setting to construct the Feller semigroup required in \cite{KS} a less restrictive assumption on $\delta$ than the one in Theorem \ref{thm1}, i.e.\,only
$$\delta<1 \wedge \left(\frac{2}{d-2}\right)^2.$$ 
In the proof of Theorem \ref{thm1} we employ the test function of \cite{KS} which seems to give the least restrictive conditions on $\delta$. 
\medskip

Generally speaking, the time-inhomogeneous case $b=b(t,x)$, i.e.\,when the vector field can have strong singularities both in time and in the spatial variables, presents the next level of difficulty compared to the time-homogeneous case since many elliptic instruments, such as deep results of semigroup theory, are no longer available.
The assertion of Theorem \ref{thm1} for $q>d$ close to $d$ (except the bound on the time derivative $\partial_tu$) and Theorems \ref{thm2}, \ref{thm3}, were proved under more restrictive assumption 
\begin{equation}
\label{osaka}
\delta<\frac{1}{d^2}
\end{equation}
 in \cite{Ki,KiM}. In \cite{Ki}, the construction of the Feller evolution family for \eqref{kolm_eq} with time-inhomogeneous form-bounded $b \in \mathbf{F}_{\delta,g}$ used a direct argument via a parabolic extension of the Moser-type iteration procedure employed in \cite{KS}, and thus did not require the Trotter approximation theorem. 

\medskip

In the present paper we improve these results: strengthen gradient bounds and relax the assumptions on $\delta$, and somewhat simplify the iteration procedure used to construct the Feller evolution family for \eqref{kolm_eq}.
In particular, the proof of gradient estimates in the present paper  (Theorem \ref{thm1}) profits from the presence of the time derivative $\partial_tu$ instead of eliminating it as in \cite{Ki}. In a sense, this refined proof of gradient estimates is truly parabolic, which allows us to impose more general assumptions on $\delta$ than in \cite{Ki}. Let us add that the control over $\partial_tu$ is a non-trivial matter if one is searching for the least restrictive assumptions on $\delta$.
As a by-product of the new proof of gradient estimates,  we obtain a new integral estimate on $\partial_t u$, see remark 2 after Theorem \ref{thm1}. The proof of Theorem \ref{thm2} is somewhat simpler than the proof of the corresponding result in \cite{Ki} because it uses interpolation inequalities differently.

\medskip

There exist counterparts of Theorems \ref{thm1}-\ref{thm3} for a larger class of singular vector fields $b$ than $\mathbf{F}_{\delta,g}$, see \cite{Ki_Morrey}. But in that approach one needs smaller $\delta$ than in the present paper. 
 Let us also mention the results in \cite{Kr} on similar gradient bounds for non-divergence form operators, with $b$ that also can have critical-order singularities, which the author of \cite{Kr} applied to establish well-posedness of the corresponding SDEs. Some gradient bounds of the same type were used in \cite{BFGM} to prove well-posedness of the stochastic transport equation with drift in the \eqref{LPS} class; we refer also to \cite{KSS} where the stochastic transport equation with form-bounded drift was treated.

\medskip

Let us return to the almost sharp condition $\delta<4$ (we refer to (a), (b) above). After this paper was written, the critical value $\delta=4$ for form-bound $b=b(x)$ was reached in \cite{Ki_Orlicz} at the level of semigroup theory, energy inequality and the existence of unique weak solution to equation \eqref{kolm_eq} considered on a compact Riemannian manifold, see also \cite{KiS_Feller_4} where an extension of these results to $\mathbb R^d$ was obtained, which required careful work with appropriate weights. Also in \cite{KiS_Feller_4}, we constructed the corresponding to \eqref{kolm_eq} strongly continuous Feller semigroup in $C_\infty$ with form-bounded $b=b(x)$, now for all $\delta<4$. The proof uses De Giorgi's method and the Trotter approximation theorem. However, it is not clear at the moment how to obtain a parabolic counterpart of the argument of \cite{KiS_Feller_4} that would allow to treat time-inhomogeneous form-bounded $b$.

\medskip

In the present paper we work in dimensions $d \geq 3$. In dimensions $d=1,2$ the Sobolev embedding theorem is stronger, so from this point of view these dimensions are simpler. However, it is also true that in dimensions $d=1,2$ there is no appropriate Hardy inequality for $b(x)=\frac{x}{|x|^2}$; in fact, this vector field is not even in $L^2_{\loc}$ in these dimensions. So, it is not an element of the class of form-bounded vector fields (the approach of \cite{Ki_Morrey} mentioned above can be used to address this issue in dimension $d=2$). 

\subsection{Notations} 
\label{notations}

1. Let $X$, $Y$ be two Banach spaces. We write $$T=s\mbox{-} Y \mbox{-}\lim_n T_n$$ for bounded linear operators $T$, $T_n \in \mathcal B(X,Y)$ if $$\lim_n\|Tf- T_nf\|_Y=0 \quad \text{ for every $f \in X$}.
$$ 

\medskip

2.~For a given Borel measurable function $f=f(t,y)$, $(t,y) \in \mathbb R \times \mathbb  R^{d}$, define mollifiers $$E^d_\varepsilon f(\tau,x)=(e^{\varepsilon\Delta_d}f)(\tau,x)=(4\pi \varepsilon)^{-\frac{d}{2}}\int_{\mathbb R^d}e^{-\frac{|x-y|^2}{4\varepsilon}}f(\tau,y)dy, \quad \varepsilon>0,$$
i.e.\,we apply heat semigroup on $\mathbb R^d$ in the spatial variables,
and 
$$E^1_\varepsilon f(\tau,x)=(e^{\varepsilon\Delta_1} f)(\tau,x)=\frac{1}{\sqrt{4\pi \varepsilon}}\int_{\mathbb R}e^{-\frac{|\tau-s|^2}{4\varepsilon}}f(s,x)ds,$$
i.e.\,we apply the one-dimensional heat semigroup in the time variable.  
It is easy to see that
 $$E_\varepsilon^{1+d}=E_\varepsilon^1E_\varepsilon^d,$$
where the left-hand side is the mollifier obtained by applying the heat semigroup on $\mathbb R^{1+d}$.

\medskip

3.~Let $\mathcal S(\mathbb R^{1+d})$ denote the Schwartz space of rapidly decreasing functions, i.e.\,functions in $C^\infty(\mathbb R^{d+1})$ which tend to zero at infinity, together with all their derivatives, faster than any reciprocal power of $|t|+|x|$, $(t,x) \in \mathbb R^{1+d}$.

\medskip

4.~For a given $b \in \mathbf{F}_{\delta,g}$, we fix a regularizing sequence $\{b_n\}\subset [L^\infty(\mathbb R^{1+d})\cap C^\infty(\mathbb R^{1+d})]^d$ of vector fields such that, for any $0<t<\infty$,
\smallskip

(\textit{i}) $$\lim_{n\rightarrow \infty} \|b_n-b \|_{L^2([0,t] \times K)}=0, \quad \text{ for every compact $K\subset \mathbb R^d$};
$$

\smallskip

(\textit{ii}) there is a sequence of functions $0 \leq g_n \subset L^1_\loc(\mathbb R)$ such that $\sup_n\int_{0}^t g_n(\tau)d\tau = c_\delta(t)<\infty;$

\smallskip

(\textit{iii}) for all $n \geq 1$, $$\int_{0}^t\|b_n(\tau)f (\tau)\|_2^2d\tau \leq \delta\int_0^t \|\nabla f(\tau)\|_2^2d\tau+\int_0^t g_n(\tau)\|f(\tau)\|_2^2d\tau \quad (f \in \mathcal S(\mathbb R^{1+d})).$$

\medskip

The collection of all such sequences will be denoted by $[b]^r$. 

\medskip

We provide one possible construction of a $\{b_n\} \in [b]^r$ in Section \ref{approx_sect}.

\medskip

\section{Main results}
Our first result concerns the classical solutions to the Cauchy problem
\begin{equation}
\label{cauchy}
\bigg(\partial_\tau - \Delta + b_n \cdot\nabla_x\bigg)u(\tau) = 0, \quad 0\leq s<\tau<\infty,\quad x\in\mathbb{R}^d, \quad u(s)=h \in C_c^1(\mathbb{R}^d).
\end{equation}

\begin{theorem}
\label{thm1} Let $b \in \mathbf{F}_{\delta,g}$.
Assume that $q>d$ and $\delta>0$ satisfy the following constraints:
\[
\sqrt{\delta}<\bigg(\sqrt{q-1}-\frac{q-2}{2}\bigg)\frac{2}{q} \quad \text{ if } q=d+\varepsilon, \qquad d=3,4,
\]
where $0<\varepsilon\leq 2(\sqrt{2}-1)$,
\[
\sqrt{\delta}<(1-\mu)\frac{q-1}{q-2}\frac{1}{q} \quad \text{ if } q=d+\varepsilon, \qquad d\geq 5,
\]
where $0<\varepsilon\leq 1$, $0<\mu<1$, $16\mu>(1-\mu)^4\frac{(q-1)^2}{(q-2)^4}$ (we discuss below some sufficient conditions for these inequalities to hold). 
Let $ \{b_n\} \in [b]^r$ be a regularizing sequence for $b$, and let $u=u_n$ be the classical solution to Cauchy problem \eqref{cauchy}. Then there are constants $C_i=C_i(q,d,\delta)>0$, $i=1,2$, independent of $n$, such that, for all $0\leq s<t<\infty$, 
\begin{equation*}
\sup_{s \leq \tau \leq t } \|\nabla u(\tau)\|_q^q +  C_1 \int_s^{t} \big\|\nabla u(\tau)\big\|_{qj}^q d \tau \leq e^{C_2c_\delta(t)} \|\nabla u(s)\|_q^q, \quad j:=\frac{d}{d-2}.
\end{equation*}
\end{theorem}

\begin{remarks}

1. In particular, the assertion of Theorem \ref{thm1} is valid for

\smallskip

{\rm(a)} $\sqrt{\delta}=\frac{1.8}{d}$, $q=d+\frac{1}{48}$ for $d=3$. Or $\sqrt{\delta}=\frac{1.4}{d}$, $q=d + 0.014$ for $d=4$.

\medskip

{\rm(b)} $\sqrt{\delta}=\frac{1}{d}$, $q=d+1$ for $d\geq 5$ (note that the result in \cite{Ki} under the assumption $\sqrt{\delta}<\frac{1}{d}$ allows to only take $q=d+\varepsilon$ for $\varepsilon>0$ sufficiently small).

\medskip

${\rm(b^\prime)}$ $\sqrt{\delta}= (1-\frac{a}{16+a})\frac{q-1}{q-2}\frac{1}{q}$, $a=\frac{(q-1)^2}{(q-2)^4}$, $q=d+\varepsilon$ for all $\varepsilon\in ]0,1]$ for $d\geq 5$. This is a slightly rougher bound on admissible $\delta$ than the one in Theorem \ref{thm1}, but it is easy to verify. 

\medskip

Let us further compare the assumptions on $\delta$ in Theorem \ref{thm1} and in the analogous result in \cite{Ki}. Let us denote
$$
c_{\rm o}(q):=\frac{1}{q^2}, \quad c_{\rm n}(q):=\biggl((1-\frac{a}{16+a})\frac{q-1}{q-2}\frac{1}{q}\biggr)^2.
$$
Then the assumption on $\delta$ in \cite{Ki} (cf.\,\eqref{osaka}) and the mentioned in $\rm (b^\prime)$ somewhat rougher than necessary assumption on $\delta$ allowing to treat $q=d+\varepsilon$ in Theorem \ref{thm1} become, respectively,
$$
\delta\leq c_{\rm o}(d+\varepsilon), \quad \delta \leq c_{\rm n}(d+\varepsilon).
$$
Now one can see right away that
$$
c_{\rm o}(d+\varepsilon)<c_{\rm n}(d+\varepsilon) \text{ for all } \varepsilon \in ]0,1],
$$
In fact, a stronger result is valid: $c_{\rm o}(d)<c_{\rm n}(d+\varepsilon)$, $\varepsilon \in ]0,1]$.

\medskip

In Figure \ref{fig1} we plot the graph of $\frac{c_n(d+\varepsilon)}{c_o(d+\varepsilon)}$ against dimension $d$ for $\varepsilon>0$ taken to be small. In this plot we extended the definitions of the upper bounds on admissible $\delta$, i.e.\,$c_{\rm o}(d+\varepsilon)$, $c_{\rm n}(d+\varepsilon)$, to $d=3,4$ using \cite{Ki} and Theorem \ref{thm1}.
\begin{center}
\begin{figure}[h]
\label{fig1}
\includegraphics[scale=0.55]{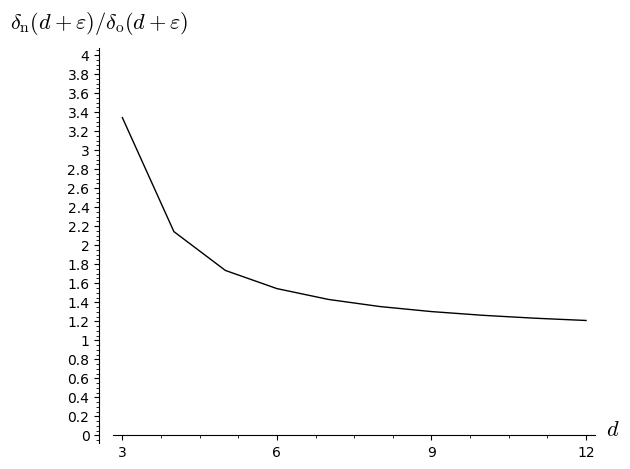}
\caption{\label{fig1}}
\end{figure}
\end{center}

\medskip

2. In the assumptions of Theorem \ref{thm1} we actually obtain a stronger regularity estimate: 
\begin{align*}
\sup_{s \leq \tau \leq t } \|\nabla u(\tau)\|_q^q &  + c_0\int_s^t\||\nabla u(\tau)|^\frac{q-2}{2}  \partial_\tau u(\tau)\|_2^2 d\tau + c_1\sum_{i=1}^d \int_s^t \big\langle |\nabla_i \nabla u(\tau) |^2, |\nabla u_n(\tau)|^{q-2} \big\rangle d \tau \\
& \leq e^{c_2(t)} \|\nabla u(s)\|_q^q.
\end{align*}
\end{remarks}

\begin{theorem}
\label{thm2}
Let $\{b_n\} \in [b]^r\subset \mathbf{F}_{\delta,g}$ with $\delta$  and $q>d$ satisfying the assumptions of Theorem \ref{thm1}. Let $u_n$ be the classical solution to Cauchy problem \eqref{cauchy}.
For every $n=1,2,\dots$, define operators $U_n^{t,s} \in \mathcal B(C_\infty)$, $0\leq s\leq t<\infty$, by
$$
U_n^{t,s}f:=u_n(t).
$$
Then the limit
$$
U^{t,s}:=s\mbox{-}C_\infty(\mathbb R^d)\mbox{-}\lim_{n}U_n^{t,s} \quad \text{(uniformly in $0 \leq s < t\leq 1$)}
$$
exists and determines a Feller evolution family on $\{(t,s)\in\mathbb R^2_+\mid 0\leq t-s<\infty\}\times C_\infty(\mathbb R^d)$.
\end{theorem}

\smallskip

\begin{remarks}
1.~In the proof of Theorem \ref{thm2}, we establish the following result that is important on its own: if $b \in \mathbf{F}_\delta$ with $\delta<4$, then solutions $\{u_n\}_{n \geq 1}$ constitute a Cauchy sequence in $L^\infty([0,1],L^r)$, $\frac{2}{2-\sqrt{\delta}}<r<\infty$.  (This result appeared previously only in the time-homogeneous case $b=b(x)$ \cite{KiS_super}.)

\medskip

2.~Theorem \ref{thm2} is proved using the following inequality:
\[
\|u_n-u_m\|_{L^\infty([0,1] \times \mathbb R^d)}\leq B \|u_n-u_m\|_{L^r([0,1] \times \mathbb R^d)}^\gamma, \quad B<\infty, \gamma>0, \quad r>\frac{2}{2-\sqrt{\delta}}.
\]
Its proof uses an iteration procedure that depends on the gradient bound of Theorem \ref{thm1}. The main concern of the iteration procedure is the strict positivity of exponent $\gamma$. 

\medskip

3.~The evolution family $\{U^{t,s}\}_{0 \leq s \leq t<\infty}$ does not depend on the choice of concrete regularization $\{b_n\} \in [b]^r$. In this sense, the ``approximation solution'' $u(t)=U^{t,s}f$  is unique. Moreover, it is readily seen that, for $f \in C_\infty \cap L^2$, $u(t)$ is a weak solution to Cauchy problem $\partial_t-\Delta +b\cdot\nabla=0$, $u(s)=f$ in the usual sense, and that it satisfies the gradient estimates in Theorem \ref{thm1} if $f \in C_\infty \cap W^{1,q}$. 
\end{remarks}

Theorem \ref{thm1} can be extended to non-homogeneous parabolic equation with form-bounded right-hand side, moreover, the corresponding gradient estimates can be localized. This, together with Theorem \ref{thm2}, allows to prove the following result by arguing as in \cite{KiM}:

\begin{theorem}
\label{thm3}
Let $b \in \mathbf{F}_{\delta,g}$ satisfy the conditions of Theorem \ref{thm1}. Then there exist probability measures $\mathbb P_x$, $x \in \mathbb R^d$ on $(C([0,T],\mathbb R^d),\sigma(\omega_r \mid 0 \leq r \leq t))$, where $\omega_t$ is the coordinate process, satisfying
$$
\mathbb E_x[f(\omega_r)]=P^{0,r}f(x), \quad 0 \leq r \leq T, \quad f \in C_\infty(\mathbb R^d),
$$
where $P^{t,r}(b):=U^{T-t,T-r}(\tilde{b})$, $\tilde{b}(t,x)=b(T-t,x)$, such that $\mathbb P_x$ is a weak solution to SDE 
\begin{equation}
\label{seq2}
X_t=x-\int_0^t b(s,X_s)ds + \sqrt{2}B_t.
\end{equation}
Moreover $\mathbb P_x$ is unique in a large class of weak solutions (see \cite{KiM}).
\end{theorem}

\bigskip

\section{Construction of a regularizing sequence $\{b_n\} \in [b]^r$}

\label{approx_sect}

We use notations introduced in Section \ref{notations}.
Set
\[
b_n(\tau,x):=E^{1+d}_{\varepsilon_n}(\mathbf{1}_{Q_n}b)(\tau,x), \qquad Q_n=[0,n]\times
 B_d(0,n).
\]
Select $\{\varepsilon_n\}\downarrow 0$ from the requirement $\lim_n\int_0^t\|E^{1+d}_{\varepsilon_n}(\mathbf{1}_{Q_n}b)(\tau)-(\mathbf{1}_{Q_n}b)(\tau)\|_2^2d\tau=0$. 

Note that $|E\phi|\leq\sqrt{E|\phi|^2}$, $|E(\phi\psi)|\leq \sqrt{E|\phi|^2}\sqrt{E|\psi|^2}$. We have (for a.e. $t>0$)
\begin{align*}
|E^d_{\varepsilon_n} \mathbf{1}_{B_d(0,n)} b(\tau,x)|^2 & \leq \langle e^{\varepsilon_n\Delta_d}(x,\cdot)\mathbf{1}_{B_d(0,n)}|b(\tau,\cdot)|^2\rangle \leq \delta \|\nabla_x\sqrt{E_{\varepsilon_n}^d(x,\cdot)}\|_2^2+g(\tau)\\
&=(4\varepsilon_n)^{-1}\delta\big\langle\frac{|x-\cdot|^2}{4\varepsilon_n}e^{\varepsilon_n\Delta_d}(x,\cdot)\big\rangle	+ g(\tau) \leq C(d)\varepsilon_n^{-1}\delta+g(\tau).
\end{align*}
Thus $|E_{\varepsilon_n}^{1+d}(\mathbf{1}_{Q_n}b)(\tau,x)|\leq \sqrt{C(d) \delta}\varepsilon_n^{-\frac{1}{2}}+\sqrt{E^1_{\varepsilon_n}(\mathbf 1_{[0,n]}g)(\tau)}$ an so $|b_n|\in L^\infty(\mathbb{R}^{1+d})$. It is  clear that $b_n$ are smooth.

Next, for $f\in\mathcal{S}(\mathbb{R}^{1+d})$, $\int_0^t\|b_nf\|_2^2=\int_0^t\|b_nf_t\|_2^2$, where  $f_t(\tau,x):=\mathbf{1}_{[0,t]}f(\tau,x)$, and
\begin{align*}
\int_0^t\|b_n f\|_2^2 &\leq\int_0^t \langle E^1_{\varepsilon_n}(\mathbf 1_{[0,n]}b^2),E^{d}_{\varepsilon_n} |f_t|^2\rangle \leq \int_{\mathbb{R}}\langle \mathbf 1_{[0,n]}b^2,E^{1+d}_{\varepsilon_n} |f_t|^2\rangle\\
& \leq \delta\int_0^n\|\nabla\sqrt{E^{1+d}_{\varepsilon_n} |f_t|^2}\|_2^2 + \int_0^n gE^1_{\varepsilon_n}\langle E^d_{\varepsilon_n} |f_t|^2\rangle \\
&\leq\delta\int_0^n E^1_{\varepsilon_n}\langle E^d_{\varepsilon_n}|\nabla |f_t||^2\rangle + \int_0^n g E^1_{\varepsilon_n}\langle E^d_{\varepsilon_n} |f_t|^2\rangle\\
& \leq\delta\int_0^n E^1_{\varepsilon_n}\|\nabla |f_t|\|_2^2 + \int_0^n g E^1_{\varepsilon_n}\|f_t\|_2^2.
\end{align*}
\begin{align*}
&\int_0^n E^1_{\varepsilon_n}\|\nabla |f_t|\|_2^2=\int_{\mathbb{R}} \mathbf{1}_{[0,n]} E^1_{\varepsilon_n}(\mathbf{1}_{[0,t]}\|\nabla |f|\|_2^2)\\
&=\int_{\mathbb{R}} (E^1_{\varepsilon_n}\mathbf{1}_{[0,n]}) \mathbf{1}_{[0,t]}\|\nabla |f|\|_2^2 \leq \int_0^t \|\nabla f\|_2^2.
\end{align*}
\begin{align*}
\int_0^n gE^1_{\varepsilon_n}\|f_t\|_2^2 \leq \int_0^t (E^1_{\varepsilon_n} g)\|f\|_2^2.
\end{align*}
Therefore
\[
\int_0^t\|b_nf\|_2^2\leq \delta\int_0^t\|\nabla f\|_2^2 +\int_0^tg_n\|f\|_2^2, \qquad g_n(\tau):=E^1_{\varepsilon_n}g(\tau).
\]
It is seen now that $\{b_n\} \in [b]^r$ assuming that $g \in L^1 + L^\infty$.

\begin{remarks}
1. In the definition of $\mathbf F_{\delta,g}$ we require only $0\leq g\in L^1_\loc(\mathbb{R})$. If we define $b_n:=\mathcal E^1_{\varepsilon_n}E^d(\mathbf{1}_{Q_n}b)$, where $\mathcal E^1_{\varepsilon_n}$ denotes K.\,Friedrichs mollifier, then arguing as above one readily concludes that $\{b_n\} \in [b]^r$ with $g_n:=\mathcal E^1_{\varepsilon_n}(\mathbf 1_{[-\varepsilon_n,t+\varepsilon_n]}g)$, $\sup_n\int_0^t g_n(\tau)d\tau\leq \int_{-1}^{t+1}g(\tau)d\tau$.

\medskip

2. A construction of $\{b_n\}$ of this type without the indicator function in the spatial variables was considered in \cite{Ki_survey, KiS_MA}.
\end{remarks}

\bigskip

\section{Proof of Theorem \ref{thm1}}

\label{proof_thm1_sect}

\begin{proof}
Denote $w=\nabla_x u(\tau,x)$, $\phi = - \nabla \cdot (w |w|^{q-2}) \equiv - \sum_{i=1}^d \nabla_i (w_i |w|^{q-2})$.
Since $b_n$ is smooth and bounded, we can multiply the equation \eqref{cauchy} by $\bar{\phi}$ and integrate by parts to obtain
\begin{equation}
\label{id}
q^{-1}\partial_\tau \|w\|_q^q + I_q + (q-2)J_q = X_q, 
\end{equation}
where
$$
I_q := \sum_{i=1}^d \big\langle |\nabla w_i |^2, |w|^{q-2} \big\rangle, \quad J_q := \big\langle |\nabla |w| |^2, |w|^{q-2} \big\rangle, \quad X_q :=\Real \langle b_n \cdot w, \nabla \cdot \big( w |w|^{q-2} \big) \rangle.
$$

\textbf{1}.~ Case $d=3$, $d=4$. Clearly, $X_q =\Real \langle b_n \cdot w, |w|^{q-2} \Delta u \rangle +(q-2)\Real \langle b_n \cdot w, |w|^{q-3} w \cdot \nabla |w| \rangle,$  
\begin{align*}
\Real \langle b_n \cdot w, |w|^{q-2} \Delta u \rangle & = \Real\langle b_n \cdot w, |w|^{q-2} (\partial_\tau u + b_n \cdot w )\rangle \\
 & = B_q + \Real\langle b_n \cdot w, |w|^{q-2} \partial_\tau u \rangle \qquad (B_q:=\langle |b_n\cdot w|^2,|w|^{q-2}\rangle)\\
 & = B_q + \Real \langle (-\partial_\tau u  + \Delta u), |w|^{q-2} \partial_\tau u \rangle \\
 & = B_q - \langle |\partial_\tau u|^2, |w|^{q-2} \rangle - q^{-1} \partial_\tau \|w\|_q^q -(q-2)\Real \langle |w|^{q-3} w \cdot \nabla|w|, \partial_\tau u \rangle ; \\
\Real\langle b_n \cdot w, |w|^{q-2} \Delta u \rangle & \leq  B_q - \langle |\partial_\tau u|^2, |w|^{q-2} \rangle - q^{-1} \partial_\tau \|w\|_q^q + (q-2) J_q^\frac{1}{2} \langle |\partial_\tau u|^2, |w|^{q-2} \rangle^\frac{1}{2} \\
 & \leq - q^{-1} \partial_\tau \|w\|_q^q + B_q + \frac{(q-2)^2}{4} J_q \\
 & \leq - q^{-1}\partial_\tau \|w\|_q^q + \bigg[\frac{q^2 \delta}{4}+\frac{(q-2)^2}{4}\bigg] J_q +g_n(\tau)\|w\|_q^q ;
\end{align*}
\begin{align*}
|\langle b_n \cdot w, |w|^{q-3} w \cdot \nabla |w| \rangle| & \leq B_q^\frac{1}{2} J_q^\frac{1}{2} \leq \bigg[\frac{1}{4 \varepsilon} B_q + \varepsilon J_q \bigg] \\
& \leq \bigg[\frac{1}{4 \varepsilon} \frac{q^2 \delta}{4} + \varepsilon \bigg]J_q + \frac{g_n(\tau)}{4 \varepsilon} \|w\|_q^q \\
& = \frac{q \sqrt{\delta}}{2} J_q + \frac{g_n(\tau)}{q \sqrt{\delta}} \|w\|_q^q \quad \quad (\varepsilon = \frac{q \sqrt{\delta}}{4} ).
\end{align*}
Thus $$X_q \leq - \frac{1}{q} \partial_\tau \|w\|_q^q + \bigg[\frac{q^2 \delta}{4}+\frac{(q-2)^2}{4} + (q-2) \frac{q \sqrt{\delta}}{2} \bigg] J_q + \big(\frac{q-2}{q\sqrt{\delta}} + 1 \big) g_n(\tau) \|w\|_q^q,$$ and hence (taking in to account $I_q\geq J_q$)
\[
\frac{2}{q}\partial_\tau\|w\|_q^q + \bigg[q -1 - \frac{q^2 \delta}{4} - \frac{(q-2)^2}{4} - (q-2) \frac{q \sqrt{\delta}}{2} \bigg] J_q \leq \bigg(\frac{q-2}{q\sqrt{\delta}} + 1 \bigg) g_n(\tau) \|w\|_q^q .
\]
Set $\mu_\tau:=\frac{q}{2}\big(\frac{q-2}{q\sqrt{\delta}}+1\big)\int_s^\tau g_n(r)dr$, so
\[
\frac{2}{q}\partial_\tau\big(e^{-\mu_\tau}\|w(\tau)\|_q^q\big) + \bigg[q -1 - \frac{q^2 \delta}{4} - \frac{(q-2)^2}{4} - (q-2) \frac{q \sqrt{\delta}}{2} \bigg] e^{-\mu_\tau} J_q(\tau) \leq 0.  \tag{$\star$} 
\]
It is readily seen that
\[
q -1 - \frac{q^2 \delta}{4} - \frac{(q-2)^2}{4} - (q-2) \frac{q \sqrt{\delta}}{2} > 0
\]
is equivalent to
\[
\sqrt{\delta}<\bigg(\sqrt{q-1}-\frac{q-2}{2}\bigg)\frac{2}{q} \text{ for } q=d+\varepsilon, \; 3\leq d\leq 6, \; 0<\varepsilon\leq 2(\sqrt{2}-1),
\]
and holds for $\sqrt{\delta}= \frac{1,8}{d}$, $q=d+ \frac{1}{48}$, $d=3$, and for $\sqrt{\delta}=\frac{1,4}{d}$, $q=d+0,014$, $d=4$. 

Finally, using the uniform Sobolev inequality and the bound $\int_0^tg_n\leq c_\delta (t)$, we obtain from $(\star)$
\[
\sup_{s \leq r \leq t } \|w(r)\|_q^q + C_1 \int_s^t \|w(\tau)\|_{qj}^q d \tau \leq e^{C_2c_\delta(t)}\|\nabla u(s)\|_q^q, \quad C_1>0.
\]
Here we have used that $ U_n^{s_1,s}u(s) = e^{(s_1-s)\Delta}u(s) - \int_s^{s_1} U_n^{s_1,\tau} b_n \cdot \nabla e^{(\tau-s)\Delta} u(s) d \tau$ and, for $s_1-\tau\leq 1$,
\[
\| \nabla U_n^{s_1,\tau} \|_{q \to q} \leq \frac{c_n}{\sqrt{s_1-\tau}},  \; \; \| \nabla \int_s^{s_i} U_n^{s_1,\tau} b_n \cdot \nabla e^{(\tau-s) \Delta} u(s) d \tau \|_q \leq 2 c_n \|b_n\|_\infty \sqrt{s_1-s} \|\nabla u(s)\|_q,
\]
 so that $\lim_{s_1 \downarrow s} \| \nabla U_n^{s_1,s} u(s)\|_q = \lim_{s_1 \downarrow s}\| \nabla e^{(s_1-s)\Delta}u(s)\|_q = \|\nabla u(s)\|_q.$

\smallskip

\textbf{2}. Case $d \geq 5$. Now we estimate the term $X_q^\prime :=\Real \langle b_n \cdot w, |w|^{q-2} \Delta u \rangle$ as follows.
\begin{align*}
X_q^\prime & = \Real\langle - \partial_\tau u + \Delta u, |w|^{q-2} \Delta u \rangle \\
& = \langle |\Delta u|^2, |w|^{q-2} \rangle - \Real\langle \partial_\tau u , |w|^{q-2} \Delta u \rangle,
\end{align*}
\begin{align*}
X_q^\prime & =\Real \langle b_n \cdot w, |w|^{q-2} (\partial_\tau u + b_n \cdot w) \rangle \\
& = B_q + \Real\langle b_n \cdot w, |w|^{q-2} \partial_\tau u \rangle.
\end{align*}
Thus,
\begin{align*} 
\langle |\Delta u|^2, |w|^{q-2} \rangle & = B_q + \Real\langle \partial_\tau u, |w|^{q-2} ( b_n \cdot w + \Delta u  )\rangle \\
& = B_q + \Real\langle \partial_\tau u, |w|^{q-2} ( - \partial_\tau u + 2 \Delta u  )\rangle\\
& = B_q - \langle |\partial_\tau u|^2, |w|^{q-2} \rangle + 2\Real \langle \partial_\tau u,|w|^{q-2} \Delta u \rangle \\
& = B_q - \langle |\partial_\tau u|^2, |w|^{q-2} \rangle - \frac{2}{q}\partial_\tau\|w\|_q^q -2 (q-2)\Real \langle \partial_\tau u,|w|^{q-3} w \cdot \nabla |w| \rangle \\
& \leq B_q - \langle |\partial_\tau u|^2, |w|^{q-2} \rangle - \frac{2}{q}\partial_\tau\|w\|_q^q + (q-2)^2 J_q + \langle |\partial_\tau u|^2,|w|^{q-2} \rangle \\
& = B_q - \frac{2}{q}\partial_\tau\|w\|_q^q + (q-2)^2 J_q;
\end{align*}

\begin{align*}
X_q^\prime & \leq \langle |\Delta u|^2, |w|^{q-2} \rangle^\frac{1}{2} B_q^\frac{1}{2} \leq \epsilon \langle |\Delta u|^2, |w|^{q-2} \rangle + \frac{1}{4 \epsilon} B_q \\
& \leq - \frac{2 \epsilon}{q}\partial_\tau \|w\|_q^q + \bigg(\epsilon + \frac{1}{4 \epsilon} \bigg) B_q + (q-2)^2 \epsilon J_q \\
& \leq - \frac{2 \epsilon}{q}\partial_\tau \|w\|_q^q + \bigg(\frac{q^2 \delta}{4} \epsilon + \frac{1}{4 \epsilon} \frac{q^2 \delta}{4} + (q-2)^2 \epsilon \bigg) J_q + \bigg(\epsilon + \frac{1}{4 \epsilon}\bigg) g_n(\tau) \|w\|_q^q\\
& (\text{here we put } \epsilon = \frac{q \sqrt{\delta}}{4} \bigg(\frac{q^2 \delta}{4}+ (q-2)^2 \bigg)^{-\frac{1}{2}})\\
& = - \frac{2 \epsilon}{q}\partial_\tau\|w\|_q^q + \frac{q \sqrt{\delta}}{2} \sqrt{\frac{q^2 \delta}{4}+ (q-2)^2} \; J_q + \bigg(\epsilon + \frac{1}{4 \epsilon}\bigg) g_n(\tau) \|w\|_q^q.
\end{align*}
Note that $X_q=X_q^\prime+X_q^{\prime\prime}$, $X_q^{\prime\prime}=(q-2)\Real\langle b_n\cdot w,|w|^{q-3}w\cdot\nabla|w|\rangle$. Estimating $X_q^{\prime\prime}$ as in Step \textbf{1}, $X_q^{\prime\prime}\leq(q-2)\bigg(\frac{q\sqrt{\delta}}{2}J_q+\frac{g_n}{q\sqrt{\delta}}\|w\|_q^q\bigg)$, we have

\[
X_q \leq - \frac{2 \epsilon}{q}\partial_\tau\|w\|_q^q + \frac{q \sqrt{\delta}}{2} \bigg( \sqrt{\frac{q^2 \delta}{4}+ (q-2)^2}  + q-2 \bigg)\; J_q + \bigg(\epsilon + \frac{1}{4 \epsilon} + \frac{q-2}{q \sqrt{\delta}} \bigg) g_n(\tau)\|w\|_q^q.
\]
Finally,
\begin{align*}
\frac{1 + 2 \epsilon}{q} \partial_\tau\|w\|_q^q &+ \bigg[ q-1 - \frac{q \sqrt{\delta}}{2} \bigg( \sqrt{\frac{q^2 \delta}{4}+ (q-2)^2} + q-2 \bigg) \bigg ] J_q\\
& \leq \bigg(\epsilon + \frac{1}{4 \epsilon} + \frac{q-2}{q \sqrt{\delta}} \bigg)g_n(\tau) \|w\|_q^q . 
\end{align*}
We are left to show that
\[
 q-1 - \frac{q \sqrt{\delta}}{2} \bigg( \sqrt{\frac{q^2 \delta}{4}+ (q-2)^2} + q-2 \bigg) > 0. \tag{$\star'$}
\]
in the assumptions (b) and ($b^\prime$).

Let $d\geq 3$, $\sqrt{\delta}=\frac{1}{d}$, $q=d+1$. It is seen that $(\star')$ is equivalent to $d>1$.

Let $d\geq 5$. Set $\sqrt{\delta}=(1-\mu)\frac{q-1}{q-2}\frac{1}{q}$, $0<\mu<1$. Then $(\star')$ takes the form
\[
q-1 -(1-\mu)\frac{q-1}{2} > (1-\mu)\frac{q-1}{2}\sqrt{1+\frac{q^2}{4(q-2)^2}(1-\mu)^2\frac{(q-1)^2}{(q-1)^2}\frac{1}{q^2}}.  
\]
The latter is equivalent to
\[
16\mu>(1-\mu)^4\frac{(q-1)^2}{(q-2)^4}
\]
which clearly follows from $16\mu\geq(1-\mu)\frac{(q-1)^2}{(q-2)^4}$. In turn the latter is equivalent to
\[
\mu\geq\frac{a}{16+a},\quad a=\frac{(q-1)^2}{(q-2)^4}.
\]
Finally, set $\mu=\frac{a}{16+a}$. 
\end{proof}
\medskip

\begin{remarks} 1. $(\star')$ fails if $d\geq 3$, $\sqrt{\delta}=\frac{1}{d}$ and $q>d+1$. 

2. It is seen that $\frac{1}{d}<(1-\frac{a}{16+a})\frac{q-1}{q-2}\frac{1}{q}$ for $q=d+\varepsilon$ and all $0<\varepsilon\leq 1$. $(\star')$ still holds for $\mu=1+\frac{8}{a} - \sqrt{(1+\frac{8}{a})^2-1}$ ($<\frac{a}{16+a}$).
\end{remarks}

\bigskip

\section{Proof of Theorem \ref{thm2}}

\begin{claim}
\label{cla1}
Let $u_n$ be the classical solution to \eqref{cauchy}. Then $\{u_n\}$ is a Cauchy sequence in $L^\infty([s,t],L^r(\mathbb{R}^d))$ for every $r\in ]\frac{2}{2-\sqrt{\delta}},\infty[$.  
\end{claim}
\begin{proof}
In Claim \ref{cla1} we allow $\delta<4$, so we do not use the gradient bounds of Theorem \ref{thm1}. Without loss of generality we will suppose that $f = \Real f$, and so $u_n$ is real, and that $r$ is a rational number (so $u_n^{r-1}$ is well defined even if $u_n$ is sign changing).

$\mathbf{(a)}$. Let $k > 2.$ Define
$$
\eta(t):=\left\{
\begin{array}{ll}
0, & \text{ if } t< k, \\
\big( \frac{t}{k} - 1 \big)^k, & \text{ if } k \leq t \leq 2 k, \quad \quad \text{ and } \zeta(x) = \eta(\frac{|x-o|}{R}), \;\; R > 0.\\
1, & \text{ if } 2 k < t,
\end{array}
\right.
$$
Note that $ |\nabla \zeta | \leq R^{-1} \mathbf{1}_{\nabla \zeta} \zeta^{1-\frac{1}{k}}$. Here $\mathbf{1}_{\nabla \zeta}$ denotes the indicator of the support of $|\nabla\zeta|$.

 Set $v := \zeta u_n(\tau)$.  Clearly,
\[
\langle \zeta (\partial_\tau -\Delta + b_{n} \cdot\nabla ) u_n(\tau), v^{r-1} \rangle =0,
\]
\[
\langle (\partial_\tau -\Delta + b_{n} \cdot \nabla) v, v^{r-1} \rangle = \big\langle [-\Delta, \zeta]_-u_n + u_n b_{n} \cdot \nabla \zeta, v^{r-1} \big\rangle, 
\tag{$\star$} 
\]
where
\begin{align*}
\langle [-\Delta, \zeta]_-u_n, v^{r-1} \rangle & =\frac{2}{r^\prime} \big\langle \nabla v^\frac{r}{2}, u_n v^{\frac{r}{2}-1} \nabla \zeta \big\rangle - \langle \nabla \zeta, v^{r-1} \cdot \nabla u_n \rangle \\
& = \frac{2}{r^\prime} \big\langle \nabla v^\frac{r}{2},v^\frac{r}{2} \frac{\nabla \zeta}{\zeta} \big\rangle - \frac{2}{r}\big\langle \frac{\nabla \zeta}{\zeta}, v^\frac{r}{2} \nabla v^\frac{r}{2} \big\rangle + \big\langle \frac{|\nabla\zeta|^2}{\zeta^2}, v^r \big\rangle\\
&=\frac{2(r-2)}{r}\big\langle \nabla v^\frac{r}{2},v^\frac{r}{2} \frac{\nabla \zeta}{\zeta} \big\rangle + \big\langle \frac{|\nabla\zeta|^2}{\zeta^2}, v^r \big\rangle.
\end{align*}
By the quadratic estimates
\begin{align*}
\big\langle u_n b_{n} \cdot \nabla \zeta, v^{r-1} \big\rangle & = \big\langle b_{n} \cdot \frac{\nabla \zeta}{\zeta}, v^r \big\rangle \\
& \leq \frac{\mu \sqrt{\delta}}{r} \|\nabla v^\frac{r}{2}\|_2^2 + \frac{r \sqrt{\delta}}{4 \mu}\big\langle \frac{|\nabla\zeta|^2}{\zeta^2}, |v|^r \big\rangle + \frac{\mu g_n(\tau)}{r \sqrt{\delta}} \|v\|_r^r \;\;( \mu > 0 ),\\
\frac{2(r-2)}{r}\big\langle \nabla v^\frac{r}{2},v^\frac{r}{2} \frac{\nabla \zeta}{\zeta} \big\rangle & \leq \frac{\mu \sqrt{\delta}}{r}\|\nabla v^\frac{r}{2}\|_2^2 + \frac{(r-2)^2}{r \mu \sqrt{\delta}}\big\langle \frac{|\nabla\zeta|^2}{\zeta^2}, |v|^r \big\rangle,
\end{align*}
we get from $(\star)$
\[
\partial_\tau \|v\|_r^r + 2\bigg( \frac{2}{r^\prime} - (1 + \mu) \sqrt{\delta} \bigg) \|\nabla v^\frac{r}{2}\|_2^2 \leq \bigg(\frac{(r-2)^2}{\mu \sqrt{\delta}} + \frac{r^2 \sqrt{\delta}}{4 \mu} +r \bigg)\big\langle \frac{|\nabla\zeta|^2}{\zeta^2}, |v|^r \big\rangle + \frac{r + \mu}{\sqrt{\delta}} g_n(\tau) \|v\|_r^r.
\]
Recalling that $\frac{2}{r^\prime} > \sqrt{\delta},$ we can find $\mu > 0$ such that $\frac{2}{r^\prime} -(1 +\mu)\sqrt{\delta} \geq 0.$ Thus
\[
\partial_\tau \|v\|_r^r \leq \bigg(\frac{4 (r-2)^2 + r^2 \delta}{4 \mu \sqrt{\delta}} +r \bigg)\big\langle \frac{|\nabla\zeta|^2}{\zeta^2}, |v|^r \big\rangle + \frac{r + \mu}{\sqrt{\delta}} g_n(\tau) \|v\|_r^r \tag{$\star\star$}
\]
Next, $\big\langle \frac{|\nabla\zeta|^2}{\zeta^2}, |v|^r \big\rangle \leq R^{-2} \| \mathbf 1_{\nabla \zeta} \zeta^{-2 \theta} |v|^r \|_1,$ $\theta := k^{-1}$. Since $\|u_n\|_\infty \leq \|f\|_\infty, \; \|\mathbf 1_{\nabla \zeta}\|_\frac{r}{2\theta} \leq c(d,\theta) R^\frac{2\theta d}{r},$ and
\[
\| \mathbf 1_{\nabla \zeta} \zeta^{-2 \theta} |v|^r \|_1 \leq \|\mathbf 1_{\nabla \zeta}u_n^{2 \theta} \|_\frac{r}{2\theta} \|v\|_r^{r-2\theta}\leq \|\mathbf 1_{\nabla \zeta} \|_\frac{r}{2 \theta} \|u_n\|_\infty^{2\theta} \|v\|_r^{r-2\theta},
\]
we obtain, using the Young inequality, the crucial estimate (\textit{on which the whole proof rests})
\[
\big\langle \frac{|\nabla\zeta|^2}{\zeta^2}, |v|^r \big\rangle \leq \frac{2 \theta}{r} [c(d)]^\frac{r}{2 \theta} R^{d-\frac{r}{\theta}} \|f\|_\infty^r + \frac{r-2\theta}{r} \|v\|_r^r.
\]
Fix $\theta$ by $\theta = \frac{r}{d+2r}.$ Now from $(\star\star)$ we obtain the inequality
\[
\partial_\tau\|v\|_r^r \leq M(r, d, \delta) R^{-\gamma} \|f\|_\infty^r + N(r, d, \delta) \|v\|_r^r, \;\; \gamma = \frac{r}{\theta}-d > 0, \tag{$\star\star\star$}
\]
from which we conclude that, for a given $\hat{\varepsilon} > 0$ there is $R$ such that $\sup_{\tau \in [s,t], n} \|\zeta u_n(\tau)\|_r \leq \frac{\hat{\varepsilon}}{2}$, and so
\[
\sup_{\tau \in [s, t], \; n, m \geq 1} \|(1_{B^c(o, 2 k R)})(u_n(\tau) - u_m(\tau))\|_r < \hat{\varepsilon}.
\]

\medskip

 $\mathbf{(b)}$. Let $k>2$. Define

$$
\eta(t):=\left\{
\begin{array}{ll}
1, & \text{ if } t< 2 k, \\
\big( 1 - \frac{1}{k} (t - 2 k) \big)^k, & \text{ if } 2 k \leq t \leq 3 k, \quad \text{and } \zeta(x) := \eta(\frac{|x-o|}{R}), \;\; R > 0.\\
0, & \text{ if } 3 k < t,
\end{array}
\right.
$$

Set $h:=u_m-u_n$. Clearly, for $r$ rational and $v=\zeta h(\tau)$,
\[
\langle(\partial_\tau h - \Delta h+b_m\cdot\nabla h),\zeta v^{r-1}\rangle=F,
\]
\[
\partial_\tau\|v\|_r^r + 4(r')^{-1}\|\nabla v^\frac{r}{2}\|_2^2 +2\langle b_{m}v^\frac{
r}{2},\nabla v^\frac{r}{2}\rangle = r F,\quad r'=\frac{r}{r-1},
\]
where
\[
F=\langle [-\Delta,\zeta]_-h,v^{r-1}\rangle+\langle (b_{n}-b_{m})\cdot\nabla u_n,\zeta v^{r-1}\rangle +\langle b_{m}\cdot\nabla \zeta,h v^{r-1}\rangle,
\]
\[
\langle [-\Delta, \zeta]_- h , v^{r-1} \rangle = \frac{2(r-2)}{r} \big\langle \nabla v^\frac{r}{2},v^\frac{r}{2} \frac{\nabla\zeta}{\zeta} \big\rangle + \big\langle \frac{|\nabla\zeta|^2}{\zeta^2}, v^r \big\rangle,
\]
\[
\big\langle \nabla v^\frac{r}{2},v^\frac{r}{2} \frac{\nabla\zeta}{\zeta} \big\rangle \leq \|\nabla v^\frac{r}{2}
\|_2\big\langle \frac{|\nabla\zeta|^2}{\zeta^2},|v|^r\big\rangle^\frac{1}{2},
\]
\[
\langle b_{m}\cdot\nabla \zeta, h v^{r-1}\rangle = \big\langle b_{m}v^\frac{r}{2}\cdot\frac{\nabla \zeta}{\zeta},v^\frac{r}{2}\big\rangle \leq \|b_{m}v^\frac{r}{2}\|_2 \big\langle \frac{|\nabla \zeta|^2}{\zeta^2},|v|^r\big\rangle^\frac{1}{2},
\]
\[
\|b_{m}v^\frac{r}{2}\|_2^2\leq \delta \|\nabla v^\frac{r}{2}\|_2^2 + g_m\|v\|_r^r.
\]
Using these estimates and fixing $\epsilon>0$ by $2r'^{-1}-(1+\epsilon)\sqrt{\delta}\geq 0$, we have
\[
\partial_\tau\|v\|_r^r+2\big(2r'^{-1}-(1+\epsilon)\sqrt{\delta}\big)\|\nabla v^\frac{r}{2}\|_2^2\leq \big(\frac{(r-2)^2}{\epsilon r}+\frac{r}{4\epsilon}+ r\big)\big\langle\frac{|\nabla \zeta|^2}{\zeta^2}|v|^r\big\rangle + (\epsilon+2)g_m\|v\|_r^r +rF_1,
\]
\[F_1=\langle\zeta|b_{n}-b_{m}|^2\rangle^\frac{1}{2}\langle \zeta |\nabla u_n|^2,  |v|^{2(r-1)}\rangle^\frac{1}{2}.
\]
Again using the estimate $\big\langle \frac{|\nabla \zeta|^2}{\zeta^2}|v|^r\big\rangle\leq MR^{-\gamma}\|f\|_\infty^r + N\|v\|_r^r$, $\gamma>0$,  and setting $\mu_\tau =N C\tau+(\epsilon+2)\int_s^\tau g_n(\nu)d\nu$, where $C=C(r,\delta)=\frac{(r-2)^2}{\epsilon r}+\frac{r}{4\epsilon}+ r$, we obtain that
\[
e^{-\mu_t}\|v(t)\|_r^r\leq \|v(s)\|_r^r + MCR^{-\gamma}\|f\|_\infty^r\int_s^te^{-\mu_\tau} d\tau +r\int_s^te^{-\mu_\tau} F_1(\tau)d\tau,
\]
\[
\|v(t)\|_r^r\leq MCR^{-\gamma}\|f\|_\infty^r e^{\mu_t}t +e^{\mu_t}r\int_s^t F_1(\tau)d\tau,
\]
\[
\int_s^t F_1(\tau)d\tau\leq \bigg(\int_0^t\langle\zeta|b_{n}-b_{m}|^2\rangle d\tau\bigg)^\frac{1}{2}\bigg(\int_s^t\langle\zeta|\nabla u_n|^2\rangle d\tau\bigg)^\frac{1}{2}(2\|f\|_\infty)^{r-1}.
\]
We estimate $\int_s^t \langle \zeta |\nabla u_n|^2 \rangle d \tau$ as follows. Note that $\langle \partial_\tau u_n -\Delta u_n + b_{n} \cdot \nabla u_n, \zeta u_n \rangle = 0$, and so
\[
\frac{1}{2} \partial_\tau \langle \zeta u_n^2 \rangle + \langle \zeta |\nabla u_n|^2 \rangle + \langle \nabla u_n, u_n \nabla \zeta \rangle + \langle b_{n} \cdot \nabla u_n, \zeta u_n \rangle = 0,
\]
\[
\partial_\tau \langle \zeta u_n^2 \rangle + \langle \zeta |\nabla u_n|^2 \rangle \leq 2 \big(\big \langle \frac{|\nabla\zeta|^2}{\zeta} \big \rangle + \langle \zeta |b_n|^2 \rangle \big) \|f\|^2_\infty,
\]
\begin{align*}
\int_s^t \langle \zeta |\nabla u_n|^2 \rangle d \tau & \leq \|\zeta u_n(s) \|_2^2  + \bigg(2 t \big\langle \frac{(\nabla \zeta)^2}{\zeta} \big \rangle + \int_0^t\langle \zeta |b_n|^2 \rangle d\tau\bigg) \|f\|^2_\infty \\
& \leq L_0(R)\| f \|_\infty^2  + (t+1) L(R) \|f\|_\infty^2, \quad L_0(R)=\|\zeta\|_2^2.
\end{align*}
(Note that $\int_0^t\langle \zeta |b_n|^2 \rangle d\tau\leq \int_0^t\langle \zeta |b|^2 \rangle d\tau+1$ for all large $n$.)

Thus, we arrived at
\begin{align*}
&\|v(t)\|_r^r\leq MCR^{-\gamma}\|f\|_\infty^r e^{\mu_t}t\\
+& \bigg(L_0(R)  + (t+1) L(R)\bigg)^\frac{1}{2}2^{r-1}\|f\|_\infty^r e^{\mu_t}r\bigg(\int_0^t\langle\zeta|b_{n}-b_{m}|^2\rangle d\tau\bigg)^\frac{1}{2}.
\end{align*}
By the definition of $b_{n}$, $\lim_{n,m}\int_0^t\langle\zeta|b_{n}-b_{m}|^2\rangle d\tau=0$, and hence for given $\hat{\epsilon}>0$ and $R=R(\hat{\epsilon})<\infty$ there is a number $P<\infty$ such that 
\[
\sup_{\tau \in [s, t], n, m \geq P}\|1_{B(o, 2 k R)}(u_n(\tau) - u_m(\tau))\|_r <  \frac{1}{2} \hat{\varepsilon}+\frac{1}{2} \hat{\varepsilon}=\hat{\varepsilon}.
\]
\end{proof}

\begin{claim}
\label{cla2}
$\{u_n\}$ is a Cauchy sequence in $L_{\infty,\infty}$.  
\end{claim} 

Here by $L_{p,r}=L_{p,r}([s,t]\times\mathbb R^d)$ we denote the Banach space of measurable functions on $[s,t]\times\mathbb R^d$ having finite norm
\[
\|v\|_{p,r}:=\bigg(\int_s^t\|v(\tau)\|_r^pd\tau\bigg)^\frac{1}{p}, \quad\|v\|_{\infty,\infty}:=\sup_{\tau\in [s,t]}\|v(\tau)\|_\infty.
\]

\medskip

\begin{proof} \textbf{1}.~Again, first we allow $\delta<4$. Note that $h(\tau)=u_m(\tau)-u_n(\tau)$ satisfies the identity
\[
(\partial_\tau -\Delta + b_m \cdot \nabla) h = (b_{n} - b_{m}) \cdot \nabla u_n,\quad h(s)=0.
\]
Multiplying the identity by $h |h|^{r-2}$, $r>\frac{2}{2-\sqrt{\delta}}$ and integrating by parts, we obtain
\[
\frac{1}{r}\partial_\tau \|v\|_2^2 + \frac{4}{r r^\prime} \|\nabla v\|_2^2 + \frac{2}{r}\Real \langle b_{m} \cdot \nabla v,v \rangle = \Real\langle (b_{n} - b_{m}) \cdot \nabla u_n, v |v|^{1-\frac{2}{r}} \rangle,
\]
where $v =h |h|^\frac{r-2}{2}$.
Now, using the quadratic estimates and the definition of class $\mathbf F_\delta$, we have 
\begin{align*}
|\langle b_{m} \cdot \nabla v,v \rangle| & \leq \varepsilon \|b_{m} v \|_2^2 + (4\varepsilon)^{-1} \| \nabla v \|_2^2\\
&\leq (\varepsilon \delta + (4 \varepsilon)^{-1} ) \|\nabla v \|_2^2 + \varepsilon g_m(\tau) \|v\|_2^2\\
& = \sqrt{\delta} \|\nabla v \|_2^2 +(2 \sqrt{\delta})^{-1} g_m(\tau) \|v\|_2^2 \quad (\varepsilon = (2 \sqrt{\delta})^{-1}, )
\end{align*}
and
\begin{align*}
|\langle (b_{n} - b_{m}) \cdot \nabla u_n, v |v|^{1-\frac{2}{r}} \rangle| & \leq \langle (|b_{n}| + |b_{m}|) |v|, |v|^{1-\frac{2}{r}} | \nabla u_n | \rangle \\
& \leq \eta \delta \|\nabla v \|_2^2  + \eta^{-1} \| |v|^{1-\frac{2}{r}} \nabla u_n \|_2^2 + \eta \frac{g_m(\tau)+g_n(\tau)}{2} \|v\|_2^2\quad (\eta>0),
\end{align*}
and hence obtain the inequality
\begin{align*}
\frac{1}{r}\partial_\tau \|v\|_2^2 & + \bigg( \frac{4}{r r^\prime} -\frac{2}{r}\sqrt{\delta} -\eta \delta \bigg) \|\nabla v\|_2^2 \\
& \leq  \eta^{-1}\||v|^{1-\frac{2}{r}} \nabla u_n\|_2^2 +  \big( r\sqrt{\delta})^{-1}g_m(\tau)+\eta\frac{g_m(\tau)+g_n(\tau)}{2}\big) \|v\|_2^2.
\end{align*}
Since $r > \frac{2}{2 -\sqrt{\delta}} \Leftrightarrow \frac{2}{r^\prime} - \sqrt{\delta} > 0,$ we can choose $k > 2$ so large that $$\frac{4}{r r^\prime} - \frac{2}{r} \sqrt{\delta} = \frac{2}{r} \big(\frac{2}{r^\prime} - \sqrt{\delta} \big)= 2 r^{-k+1}.$$
Fix $\eta$ by $$\eta \delta = \frac{4}{r r^\prime} - \frac{2}{r} \sqrt{\delta} - r^{-k+1} \;\; ( = r^{-k+1} ).$$  Thus
\begin{align*}
\partial_\tau \|v\|_2^2 & + r^{-k} \| \nabla v \|_2^2 \\
& \leq \delta r^{k-1} \| |v|^{1-\frac{2}{r}} \nabla u_n \|_2^2 +(\delta^{-\frac{1}{2}}g_m(\tau)+\delta^{-1}r^{-k+2}(g_m(\tau)+ g_n(\tau)\big) \|v\|_2^2.
\end{align*}
So, multiplying this inequality by $e^{-\mu_\tau}$, $\mu_\tau:=(\delta^{-\frac{1}{2}}+\delta^{-1})\int_s^\tau (g_m(s')+g_n(s'))ds'$, integrating over $[s,t]$, and then using the inequality $$\mu_\tau\leq \bar{\mu}_t:=2(\delta^{-\frac{1}{2}}+\delta^{-1})c_\delta(t)$$ we obtain
\[
\sup_{s \leq \tau\leq t}\|v(\tau)\|_2^2  + r^{-k}\int_s^t \| \nabla v(\tau) \|_2^2 d\tau  \leq r^ke^{\bar{\mu}_t} \int_s^t \| |v|^{1-\frac{2}{r}}(\tau) \nabla u_n(\tau) \|_2^2 d\tau. 
\]
 From the last inequality we obtain, using uniform Sobolev inequality $c_d^{-1} \|v \|^2_{2 j} \leq \| \nabla v \|_2^2$ and H\"older's inequality:
\begin{align*}
c_d r^k \sup_{s\leq \tau \leq t}\|v(\tau)\|_2^2 + &\int_s^t \|v(\tau) \|_{2j}^2 d\tau \leq c_d r^{2k} e^{\bar{\mu}_t} \int_s^t \| |v|^{1-\frac{2}{r}} \nabla u_n \|_2^2 d\tau\\
&\leq c_dr^{2k} e^{\bar{\mu}_t}\int_s^t \|\nabla u_n\|_{2x}^2 \||v|^{1-\frac{2}{r}}\|_{2x'}^2 d\tau, \quad x>1, \quad x':=\frac{x}{x-1}.
\end{align*}
\textbf{2}. Now let $d$, $\delta$ and $q>d$ satisfy the assumptions of Theorem \ref{thm1}. Thus
$$
\sup_{s\leq\tau\leq t}\|\nabla u(\tau)\|_q^2\leq e^{2C_2q^{-1}c_\delta(t)} \|\nabla u(s)\|_q^2.
$$
Selecting $x:=\frac{q}{2}$ and putting $C_3=4\delta^{-1}c_\delta +2C_2q^{-1}$, we obtain
\[
c_dr^k \|h\|_{\infty,r}^r + \|h\|^r_{r,rj}\leq c_dr^{2k}e^{C_3c_\delta(t)} \|\nabla u(s)\|_q^2 \int_s^t\|h\|^{r-2}_{x'(r-2)} d\tau.
\]
Set $D:=c_d e^{C_3c_\delta(t)} \|\nabla u(s)\|_q^2$. Then the last inequalities take the form
\[
c_dr^k\|h\|^r_{\infty,r} +  \|h\|^r_{r,rj}\leq Dr^{2k} \|h\|^{r-2}_{r-2,x'(r-2)}. \tag{$\star$}
\]
Let us use first H\"older and then Young inequalities:
\[
\|h\|^r_{\frac{r}{1-\beta},\frac{rd}{d-2+2\beta}}\leq \|h\|^{\beta r}_{\infty,r}\|h\|_{r,rj}^{(1-\beta)r}\leq \beta \|h\|^r_{\infty,r}+(1-\beta)\|h\|_{r,rj}^r, \quad 0<\beta<1.
\]
Therefore, we obtain from $(\star)$ the inequalities
\[
\|h\|_{\frac{r}{1-\beta},\frac{rd}{d-2+2\beta}}\leq D^\frac{1}{r}(r^\frac{1}{r})^{2k}  \|h\|^{1-\frac{2}{r}}_{r-2, x'(r-2)}. 
\]
Let $d\geq 5$, $\sqrt{\delta}=d^{-1}$ and $q=d+1$.
Define $\beta =\frac{2}{d^2+d+2}$, $j_1=\frac{d}{d-2+2\beta}$ and $\mathfrak{t}=\frac{j_1}{x'}$. Then $\mathfrak{t}=\frac{1}{1-\beta}$. In other cases we 
select $\beta\in ]0,q-d]$ such that $\mathfrak{t}=\frac{1}{1-\beta}$.
Thus,
\[
\|h\|_{\mathfrak{t}r,j_1r}\leq D^\frac{1}{r}(r^\frac{1}{r})^{2k}  \|h\|^{1-\frac{2}{r}}_{r-2, x'(r-2)}.
\]
Fix $r_0>\frac{2}{2-\sqrt{\delta}}$.
Successively setting $x'(r_1-2)=r_0$, $x'(r_2-2)=j_1r_1$, $x'(r_3-2)=j_1r_2,\dots$, so that $$r_n = (\mathfrak{t} -1)^{-1} \bigg( \mathfrak{t}^n \big( \frac{r_0}{x^\prime} +2\big) - \mathfrak{t}^{n-1} \frac{r_0}{x^\prime} -2 \bigg),$$ we obtain from the last inequality that
\[
\|h\|_{\mathfrak{t}r_n,j_1r_n}\leq D^{\alpha_n}\Gamma_n\|h\|^{\gamma_n}_{\frac{r_0}{x'},r_0},
\]
where
\begin{align*}
\alpha_n = &\frac{1}{r_1}\bigg(1 - \frac{2}{r_2} \bigg) \bigg(1 - \frac{2}{r_3} \bigg) \dots \bigg(1 - \frac{2}{r_n} \bigg) + \frac{1}{r_2}\bigg(1 - \frac{2}{r_3} \bigg) \bigg(1 - \frac{2}{r_4} \bigg)\dots \bigg(1 - \frac{2}{r_n} \bigg)\\
& + \dots + \frac{1}{r_{n-1}} \bigg(1 - \frac{2}{r_n} \bigg) + \frac{1}{r_n} ;\\
\gamma_n = & \bigg(1 - \frac{2}{r_1} \bigg) \bigg(1 - \frac{2}{r_2} \bigg) \dots \bigg(1 - \frac{2}{r_n} \bigg) ;\\
\Gamma_n =& \bigg[r_n^{r_n^{-1}} r_{n-1}^{r_{n-1}^{-1}(1-2 r_n^{-1})} r_{n-2}^{r_{n-2}^{-1} (1-2 r_{n-1}^{-1}) (1-2 r_n^{-1})} \dots r_1^{r_1^{-1} (1-2 r_2^{-1}) \dots (1- 2 r_n^{-1}) } \bigg]^{2k} .
\end{align*}
Since $\alpha_n = (\mathfrak{t}^n - 1)r_n^{-1}(\mathfrak{t}-1)^{-1}$ and $\gamma_n = r_0 \mathfrak{t}^{n-1}(x^\prime r_n)^{-1},$
\[
\alpha_n \leq \alpha \equiv \bigg( \frac{r_0}{x^\prime} +2 -  \frac{r_0}{j_1} \bigg)^{-1}=\frac{j_1}{r_0}\bigg(\mathfrak{t}-1+2\frac{j_1}{r_0}\bigg)^{-1}, 
\]
and
\[
\inf_n \gamma_n > \gamma = \frac{r_0}{x'}\big(\frac{r_0}{x'}+\frac{2\mathfrak{t}}{\mathfrak{t}-1}\big)^{-1} > 0, \quad \quad \sup_n \gamma_n < 1.
\]

Also, since
$$
\Gamma_n^\frac{1}{2 k} = r_n^{r_n^{-1}} r_{n-1}^{\mathfrak{t} r_n^{-1}} r_{n-2}^{\mathfrak{t}^2 r_n^{-1}} \dots r_1^{\mathfrak{t}^{n-1}r_n^{-1}}
$$
and $
b \mathfrak{t}^n \leq r_n \leq a \mathfrak{t}^n$, where $a= r_1(\mathfrak{t}-1)^{-1}, \; b = r_1 \mathfrak{t}^{-1},$ we have
\begin{align*}
\Gamma_n^\frac{1}{2 k} & \leq (a \mathfrak{t}^n)^{(b\mathfrak{t}^n)^{-1}} (a \mathfrak{t}^{n-1})^{(b\mathfrak{t}^{n-1})^{-1}} \dots (a \mathfrak{t})^{(b \mathfrak{t})^{-1}} \\
& = \bigg[ a^{(1 -\mathfrak{t}^{-n})(\mathfrak{t}-1)^{-1}} \mathfrak{t}^{\sum_{i=1}^n i \mathfrak{t}^{-i}} \bigg]^\frac{1}{b} \leq \bigg[ a^{(\mathfrak{t}-1)^{-1}} \mathfrak{t}^{\mathfrak{t}(\mathfrak{t}-1)^{-2}} \bigg]^\frac{1}{b}.
\end{align*}
Finally, note that $\|h\|_{r_0,r_0} \to 0$ as $n, m \rightarrow \infty$, and so $\|h\|_{\frac{r_0}{x'},r_0}^{\gamma_n} \leq (t-s)^{\frac{\gamma_n}{r_0(x-1)}} \|h\|_{r_0,r_0}^\gamma$ for all large $n, m$. 

Define $\nu(\tau)=\tau^{\frac{\gamma}{r_0(x-1)}}$ if $0<\tau\leq 1$ and $\tau^{\frac{1}{r_0(x-1)}}$ if $\tau>1$.

Therefore, we conclude that there are constants $B<\infty$ and $\gamma>0$ such that the following inequality is valid
\[
\|h\|_{\infty,\infty}\leq B(t-s) \|h\|_{r_0,r_0}^\gamma, \quad B(t-s)=B\nu(t-s)e^{\alpha C_3\|g^{\prime\prime}\|_\infty t}.
\]
\end{proof}
To end the proof of Theorem \ref{thm2} we note that $\|h\|_{L^{r_0}([s,t] \times \mathbb R^d)}\rightarrow 0$ uniformly in $s \in [0,t]$ according to Claim \ref{cla1}.


\bigskip

\begin{thebibliography}{99}


\bibitem{BKRS} V.\,Bogachev, N.V.\,Krylov, M.\,R\"{o}cker and S.\,Shaposhnikov, Fokker-Planck-Kolmogorov Equations, \textit{Amer.\,Math.\,Soc.}


\bibitem{NU} A.\,I.\,Nazarov and N.\,N.\,Uraltseva, The Harnack inequality and related properties for solutions to elliptic and parabolic equations with divergence-free lower order coefficients, {\em Algebra i Analiz}, \textbf{23} (2011), 136-168.


\bibitem{Ki_survey} D.\,Kinzebulatov, Form-boundedness and SDEs with singular drift, \textit{INdAM Meeting 2022: Kolmogorov Operators and Their Applications} (2024) (arXiv:2305.00146).


\bibitem{KiM} D.\,Kinzebulatov and K.R.\,Madou, Stochastic equations with time-dependent singular drift, {\em J.\,Differential Equations}, \textbf{337} (2022), 255-293 (arXiv:2105.07312).


\bibitem{KiS_super} D.\,Kinzebulatov and Yu.\,A.\,Sem\"{e}nov, On the theory of the Kolmogorov operator in the spaces $L^p$ and $C_\infty$, {\em Ann.~Scuola~Norm.~Sup.~Pisa Cl. Sci. (5)} \textbf{21} (2020), 1573-1647 (arXiv:1709.08598).


\bibitem{S} Yu.\,A.\,Sem\"{e}nov, \newblock Regularity theorems for parabolic equations, \newblock {\em J.\,Funct.\,Anal.}, \textbf{231} (2006), 375-417.


\bibitem{BFGM} L.~Beck, F.~Flandoli, M.~Gubinelli and M.~Maurelli, 
\newblock Stochastic ODEs and stochastic linear PDEs with critical
drift: regularity, duality and uniqueness. 
\newblock{\em Electron. J. Probab.}, \textbf{24} (2019),  Paper No. 136, 72 pp (arXiv:1401.1530).

\bibitem{KiS_sharp} D.\,Kinzebulatov and Yu.\,A.\,Sem\"{e}nov, { Sharp solvability for singular SDEs}, {\em Electron. J. Probab.} \textbf{28} (2023), Paper No.\,69, 1-15 (arXiv:2110.11232).


\bibitem{KS} V.\,F.\,Kovalenko and Yu.\,A.\,Sem\"{e}nov,
{\newblock $C_0$-semigroups in $L^p(\mathbb R^d)$ and $C_\infty(\mathbb R^d)$ spaces generated by differential expression $\Delta+b\cdot\nabla$.} 
(Russian) {\em Teor. Veroyatnost. i Primenen.}, \textbf{35} (1990), 449-458; translation in {\em Theory Probab. Appl.} \textbf{35} (1991), 443-453.


\bibitem{Ki} D.\,Kinzebulatov,
\newblock Feller evolution families and parabolic equations with form-bounded vector fields,
\newblock \textit{Osaka J.\,Math.}, \textbf{54} (2017), 499-516 (arXiv:1407.4861).


\bibitem{Ki_Morrey} D.\,Kinzebulatov, Parabolic equations and SDEs with time-inhomogeneous Morrey drift, arXiv:2301.13805.

\bibitem{Kr} N.\,V.\,Krylov, \newblock On strong solutions of It\^{o}'s equations with $A \in W^{1,d}$ and $B \in L^d$, {\em Ann. Probab.} \textbf{49} (2021), no. 6, 3142-3167 (arXiv:2007.06040).



\bibitem{KSS} D.\,Kinzebulatov, Yu.\,A.\,Sem\"{e}nov and R.\,Song, Stochastic transport equation with singular drift, {\em Ann.\,Inst.\,Henri Poincar\'{e} (B) Probab. Stat.}, (2024), vol.\,\textbf{60}, no. 1, 731-752 (arXiv:2102.10610).





\bibitem{Ki_Orlicz} D.\,Kinzebulatov, Laplacian with singular drift in a critical borderline case, arXiv:2309.04436.





\bibitem{KiS_Feller_4} D.\,Kinzebulatov and Yu.A.\,Sem\"{e}nov, Feller generators with drifts in  the critical range, arXiv:2405.12332.





\bibitem{KiS_MA} D.\,Kinzebulatov and Yu.\,A.\,Sem\"{e}nov, Heat kernel bounds for parabolic equations with singular (form-bounded) vector fields, {\em Math. Ann.}, \textbf{384} (2022), 1883-1929 (arXiv:2103.11482).









\end{thebibliography}
\end{document}